\documentclass{article}%
\usepackage{amsmath}
\usepackage{amsfonts}
\usepackage{amssymb}
\usepackage{graphicx,color}%
\providecommand{\U}[1]{\protect\rule{.1in}{.1in}}
\newtheorem{theorem}{Theorem}

\newtheorem{claim}[theorem]{Claim}

\newtheorem{corollary}[theorem]{Corollary}

\newtheorem{definition}[theorem]{Definition}

\newtheorem{proposition}[theorem]{Proposition}

\newenvironment{proof}[1][Proof]{\noindent\textbf{#1.} }{\ \rule{0.5em}{0.5em}}

\DeclareMathOperator*{\argmin}{arg\,min}

\newcommand{\norm}[1]{\|#1\|_{{\mathbb L}_2}}
\newcommand{\normn}[1]{\|#1\|_n}

\begin{document}
\title{
Concentration inequalities and cut-off phenomena for penalized model selection
within a basic Rademacher framework }

\author{
Pascal Massart \\
Institut de Math\'ematique d'Orsay \\
B\^atiment 307 \\
Universit\'e Paris-Saclay \\
91405 Orsay-Cedex, France 
\and
Vincent Rivoirard \\
CEREMADE \\
Universit\'e Paris Dauphine \\
Place du Mar\'echal de Lattre de Tassigny \\
75016 Paris, France\\
\&\\
 Universit\'e Paris-Saclay,\\ CNRS, Inria, Laboratoire de math\'ematiques d'Orsay,\\ 91405, Orsay, France
}

\setcounter{page}{0}

\maketitle

\begin{abstract}
This article exists first and foremost to contribute to a tribute to Patrick Cattiaux. One of the two authors has known Patrick Cattiaux for a very long time, and owes him a great deal. If we are to illustrate the adage that life is made up of chance, then what could be better than the meeting of two young people in the 80s, both of whom fell in love with the mathematics of randomness, and one of whom changed the other's life by letting him in on a secret: if you really believe in it, you can turn this passion into a profession. By another happy coincidence, this tribute comes at just the right time, as Michel Talagrand has been awarded the Abel prize. The temptation was therefore great to do a double. Following one of the many galleries opened up by mathematics, we shall first draw a link between the mathematics of Patrick Cattiaux and that of Michel Talagrand. Then we shall show how the abstract probabilistic material on the concentration of product measures thus revisited can be used to shed light on cut-off phenomena in our field of expertise, mathematical statistics. Nothing revolutionary here, as everyone knows the impact that Talagrand's work has had on the development of mathematical statistics since the late 90s, but we've chosen a very simple framework in which everything can be explained with minimal technicality, leaving the main ideas to the fore. 
\end{abstract}

\section{Introduction}
Talagrand's work on concentration of measure gave a decisive impetus to this subject, not only in its fundamental aspects, but also in 
its implications for other fields, such as statistics and machine learning, which will be of particular interest to us here. There is 
such an overflow of results that we thought it would be useful to highlight a few key ideas within a deliberately simple and 
streamlined framework. Our ambition is to illustrate why and how concentration inequalities come into play to understand cut-off 
phenomena in some high-dimensional statistical problems, while giving an insight into how these tools are constructed. 

The statistical framework in which we are going to place ourselves is that of linear regression, for which one observes a random vector 
$$Y=f+\sigma\epsilon$$
in the Euclidean space $\mathbb R^n$, where $f$ is an unknown vector to be estimated and the noise level $\sigma$ is assumed to be known at first. The noise vector $\epsilon$ has independent and identically distributed components. They are assumed to be centered in expectation and normalized, i.e. with variance equal to $1$. To keep things as simple as possible, we shall even assume that they are Rademacher variables, i.e. uniformly distributed random signs. Of course, this assumption is merely a convenience to lighten the presentation, but it is clear that everything we present here extends immediately to the case where the noise variables are bounded in absolute value by a constant $M$ (greater than $1$, of course, since the variables have a variance equal to $1$). 
This a priori parametric estimation problem can easily be turned into a non-parametric one if we bear in mind that $f$ is nothing but the vector of signal intensities on $[0,1]$ at successive instants $i/n$, with $1\leq i\leq n$. In high dimension, i.e. if $n$ tends to infinity, estimating a smooth signal is more or less the same as estimating the vector $f$. A flexible strategy for this is to select a least-squares estimator from a family given a priori. A simple illustration of this strategy is to start from an ordered basis $(\phi_j)_{1\leq j\leq n}$ of $\mathbb R^n$, then form the linear model family $S_D$, with $1\leq D\leq n$, where $S_D$ is the vector space spanned by the first $D$ basis functions $(\phi_j)_{1\leq j\leq D}$. In the case where $f$ comes from a signal, the natural ordered basis comes from the Fourier basis, as will be detailed in section 5 of the article. Selecting a proper model $S_D$ and therefore a proper least-squares estimator $\hat f_D$ built on $S_D$ for $f$ can be performed by minimizing the penalized least-squares criterion 
$$\operatorname{crit}(D)=-\parallel \hat{f}_{D}\parallel ^{2}+\operatorname{pen}\left(
D\right).$$
Choosing a penalty function of the form $\operatorname{pen}\left(
D\right)=\kappa \sigma^2 D$, the following interesting high-dimensional cut-off phenomenon can be observed (on simulations and on real data as well) on the behavior of $\hat D_\kappa =\operatorname{argmin} _D-\parallel \hat{f}_{D}\parallel ^{2}+\kappa \sigma^2 D$: if $\kappa$ stays below some critical value $\hat D_\kappa$ takes large values while at this critical value $\hat D_\kappa$ suddenly drops to a much more smaller value. Interestingly this phenomenon can be observed for a variety of penalized model selection criteria  and it helps to calibrate these criteria from the data themselves. See \cite{AB09, AB11, AM09, minpen, mas-stflour,  Sau13} for regression models,
\cite{BLPR11, LM16, Ler12, LMaRB16,  mas-stflour, RBRTM11} for density estimation, or \cite{GvH13, LT16, RBR10} for more involved models. We also refer the reader to \cite{Arl19} for a survey on minimal penalties related to the slope heuristics.
 
In this paper, we shall provide a complete mathematical analysis which shows that this phenomenon occurs with high probability with a critical value which is asymptotically equal to $1$. Since the same result has been proved in \cite{minpen} for standard Gaussian errors this shows the robustness of the phenomenon to non-Gaussianity. The reason why concentration inequalities play a crucial role in the mathematical understanding of this phenomenon is particularly clear if we consider the case of pure noise where $f=0$ and if we assume that $\sigma=1$ to make things even simpler. In this case $Y=\epsilon$ and the square norm of the least-squares estimator $\hat{f}_{D}$ can be explicitly computed as 
$$ \parallel\hat{f}_{D}\parallel^2=\sum_{1\leq j\leq D}\langle \epsilon,\phi_j\rangle^2.$$
In the Gaussian case, the summands are independent and this quantity is merely distributed according to a chi-square distribution with $D$ degrees of freedom. In the Rademacher case it is no longer the case and it is not that obvious to get a sharp probabilistic control. The very simple trick which allows to connect the issue of understanding the behavior of this quantity with Talagrand's works on concentration of product measures is to consider $\parallel\hat{f}_{D}\parallel$ rather than its square, just because of the formula $$\parallel\hat{f}_{D}\parallel=\sup _{b\in B}\langle b, \epsilon\rangle$$
where $B=\{\sum_{1\leq j\leq D}\theta_j\phi_j \vert \sum_{1\leq j\leq D}\theta_j^2\leq 1\}$. This formula allows to interprete $\parallel\hat{f}_{D}\parallel$ as the supremum of a Rademacher process. For such a process, one can apply different related techniques to ultimately obtain concentration inequalities of $\parallel\hat{f}_{D}\parallel^2$ around $D$. This concentration of $\parallel\hat{f}_{D}\parallel^2/D$ around $1$ provides an explanation for the asymptotic behavior of the critical value for $\kappa$ mentioned above. 

 The paper is organized in two parts: the first part is probabilistic while the second one is devoted to statistics. In the first part, we first revisit the connection between optimal transportation  and Talagrand's geometrical approach to concentration. In particular we emphasize the importance of the variational formula for entropy to establish this connection. Needless to say, the variational formula plays an important role in statistical mechanics and this topic as well as optimal transportation are among the topics of interest of Patrick Cattiaux. These abstract results are used to build explicit concentration bounds for suprema of Rademacher processes. In the second part, we prove two complementary results on penalized least-squares model selection that highlight the above-mentioned cut-off phenomenon. Finally, we illustrate the advantages of this approach in a non-parametric estimation context and produce a few simulations that allow us to visualize the cut-off phenomenon. 

\section{Transportation and Talagrand's convex distance for product measures}

The aim of this probabilistic part of the article is twofold. Firstly, we wish to demonstrate, in an elementary way, the link between 
transport inequalities - one of Patrick Cattiaux's mathematical interests - and concentration inequalities. In passing, we shall 
revisit the connection between the functional point of view and the isoperimetric point of view developed by Talagrand in his works on 
concentration of measure.
Secondly, the resulting concentration inequalities will be used to control the suprema of Rademacher processes.  
This will prove to be the crucial tool for the statistical part of the article.
The point of view adopted here is to focus our attention on concentration inequalities for
a function of independent random variables $\zeta(X_1,X_2,\dots, X_n)$, where $\zeta$ denotes some real 
valued measurable function on some abstract product space 
$\mathcal{X}^{n}=\mathcal{X}_{1}\times\mathcal{X}_{2}\times \dots \times \mathcal{X}_{n}$ equipped with some product $\sigma$-field 
$\mathcal{A}^{n}=\mathcal{A}_{1}\otimes\mathcal{A}_{2}\otimes \dots \otimes \mathcal{A}_{n}$.
More precisely we shall consider the following regularity condition. If $v$ denotes a positive real number, we say that $\zeta$ satisfies
the bounded differences condition in quadratic mean ($\mathcal{C}_{v}$) if 

\begin{equation}
\zeta\left(  x\right)  -\zeta\left(  y\right)  \leq\sum_{i=1}^{n}c_{i}\left(  x\right)  1\!\mathrm{l}_{x_{i}\neq y_{i}}  \text{.} \label{el1lip}%
\end{equation}
where the coefficients $c_i$'s are measurable and $$\left\Vert
{\displaystyle\sum\limits_{i=1}^{n}}
c_{i}^{2}\right\Vert _{\infty}\leq v.$$
The strength of this condition is that no structure is needed to formulate it. However, if one wants to figure out what it means, it is interesting to realize that if $\zeta$ is a smooth (continuously differentiable) convex function on $[0,1]^n$, for instance, then $\zeta$ satisfies ($\mathcal{C}_{v}$) whenever $\parallel\parallel\nabla f\parallel^2\parallel_\infty\leq v$.  In this spirit, this regularity condition will typically enable us to study the behavior of suprema of Rademacher processes, which, as announced above, will be our target example here. But the fact that no structure is required is important because it also allows to study many examples of functions of independent random variables from random 
combinatorics. It was in this field that Talagrand's early work had its most immediate impact. The contribution of Talagrand's 
seminal work in \cite{Tal95} in this context is to relax the bounded differences condition used in Mac Diarmid's bound \cite{Mcdiar}, which involves 
$
{\displaystyle\sum\limits_{i=1}^{n}}
\left\Vert c_{i}^{2}\right\Vert _{\infty}$ instead of $\left\Vert
{\displaystyle\sum\limits_{i=1}^{n}}
c_{i}^{2}\right\Vert _{\infty}$.
\subsection{Talagrand's convex distance}
The crucial concept introduced by Talagrand to make this breakthrough is what he called the convex distance which can defined as 
follows.
For any measurable set $A$ and any point $x$ in $\mathcal X^n$ let
\begin{equation}
d_T(x,A)=\sup_{\alpha \in \mathcal B_n^+}\inf_{y\in A}\sum_{i=1}^n\alpha _i1\!\mathrm{l}_{x_{i}\neq y_{i}}\label{dtal}%
\end{equation}
where $\mathcal B_n^+$ denotes the set of vectors of the unit closed euclidean ball of ${\mathbb R}^n$ with non negative components. 

As explained very well in Michel Ledoux's fine 
article \cite{Led01}, for example, the concentration 
property of a probability measure on a metric space 
results in concentration inequalities of Lipschitz 
functions around their median. 
The nice thing is that an analogous mechanism can 
be set up for Talagrand's convex distance on a 
product probability space, 
with the ($\mathcal{C}_{v}$) condition replacing the 
Lipschitz condition.  
More precisely, the role played by $d_T$ in the 
study of functions satisfying to condition 
($\mathcal{C}_{v}$) is as follows.
Assume that $v=1$ for simplicity. Choosing $A$ as a 
level set of the function $\zeta$, i.e. $A=\{\zeta\leq s\}$, 
we notice that 
$$\inf_{y\in A}\sum_{i=1}^n c_i(x)1\!\mathrm{l}_{x_{i}\neq y_{i}}\leq d_T(x,A)$$
and therefore, if $d_T(x,A)<t$, there exists some point $y$ such that $\zeta(y)\leq s$ and 
$$\sum_{i=1}^n c_i(x)1\!\mathrm{l}_{x_{i}\neq y_{i}}<t \text{.}$$
Using such a point in condition ($\mathcal{C}_{1}$) leads to $\zeta(x)<t+s$. In other words, for a function $\zeta$ satisfying 
condition ($\mathcal{C}_{1}$), the following inclusion between level sets holds true for any real number $s$ and any non negative real 
number $t$:
\begin{equation}
\{\zeta\geq s+t\}\subseteq \{d_T(.,\{\zeta\leq s\})\geq t\} \label{levelset} \text{.}
\end{equation}
This means that in terms of level sets, everything works as if $d_T$ were really a usual distance between points and $\zeta$ were a 
$1$-Lipschitz function with respect to $d_T$. Given some random vector $X$ taking its values in $\mathcal X^n$ we can now connect the 
concentration of $\zeta(X)$ around its median $M$ to the concentration rate of the probability distribution of $X$ on $\mathcal X^n$ with 
respect to $d_T$ defined as  
$$\rho (t)=\sup_{A}P\{X\in A\}P\{d_T(X,A)\geq t\}\text{,}$$
where the supremum in the formula above is extended to all mesurable sets $A$. Indeed (\ref{levelset}) leads to
$$P\{\zeta(X)\leq s\}P\{\zeta(X)\geq s+t\}\leq \rho (t)$$
so that, given a median $M$ of $\zeta(X)$, using this inequality with $s=M$ or $s=M-t$ alternatively, implies that
\begin{equation}
P\{\zeta(X)\leq M-t\}\vee P\{\zeta(X)\geq M+t\}\leq 2\rho (t). \label{medianconc}
\end{equation}
The remarkable thing is that for independent variables $X_1,X_2,\dots X_n$, Talagrand's convex distance inequality provides a universal sub-gaussian control of~$\rho$.
\begin{theorem} 
Let $X_1,X_2,\dots,X_n$ be independent random variables and set $X=(X_1,X_2,\dots, X_n)$, for all non negative real number $t$
$$\sup_{A}P\{X\in A\}P\{d_T(X,A)\geq t\}\leq \exp {(-t^2/4)}$$
where the supremum is taken over all measurable subsets of $\mathcal X^n$.
\end{theorem}
It is easy to relax the normalization constraint on the function $\zeta$ that we have used above. Given some function $\zeta$ 
satisfying to condition ($\mathcal {C}_v$) and combining Talagrand's convex distance inequality with 
inequality (\ref{medianconc}) (used for $\zeta/\sqrt{v}$ instead of $\zeta$) leads to the following immediate consequence.
\begin{corollary} 
\label{medconc}
Let $\zeta$ satisfying regularity condition ($\mathcal {C}_v$), and $X_1,X_2,\dots,X_n$ be independent random variables. Setting $Z=\zeta(X_1,X_2,\dots,X_n)$, if $M$ is a median of $Z$, then for all non negative real number $t$
$$P\{Z-M\leq -t\}\vee P\{Z-M\geq t\}\leq 2\exp (-t^2/(4v)).$$
\end{corollary}
Of course, from this concentration inequality of $\zeta(X_1,X_2,\dots,X_n)$ around the median, it is possible to deduce a concentration 
inequality around the expectation, possibly with slightly worse constants. As a matter of fact, it is better to take an alternative
route. More precisely, starting from a proper transportation inequality one can prove a concentration inequality around the expectation 
under the same regularity condition as above with neat constants. As a bonus we shall see that it will also provide a simple proof of 
Talagrand's convex distance inequality. 
\subsection{Marton's transportation inequality in action}
The link between optimal transportation and concentration has been pointed out by Katalin Marton in a series of papers (see \cite{Marton86},\cite{Marton96} and \cite{Marton96b}). Let us give a few lines of explanation based on the variational formula for entropy. Let $Q$ be some
probability distribution which is absolutely continuous with respect to
$P^{n}$. Let $\mathbb{P}$ be some probability distribution, coupling $P^{n}$
to $Q$, which merely means that it is a probability distribution on with first marginal $P^n$ and second marginal $Q$. Let $\zeta$ be some function satisfying condition ($\mathcal{C}_{v}$). Then we may write%
\[
E_{Q}\left(  \zeta\right)  -E_{P^{n}}\left(  \zeta\right)  =\mathbb{E}_{\mathbb{P}%
}\left[  \zeta\left(  Y\right)  -\zeta\left(  X\right)  \right]  \leq%
{\displaystyle\sum\limits_{i=1}^{n}}
\mathbb{E}_{\mathbb{P}}\left[  c_{i}\left(  Y\right)  \mathbb{P}\left\{
X_{i}\neq Y_{i}\mid Y\right\}  \right],
\]
which implies by applying Cauchy-Schwarz inequality twice%
\begin{align*}
E_{Q}\left(  \zeta\right)  -E_{P^{n}}\left(  \zeta\right)   &  \leq%
{\displaystyle\sum\limits_{i=1}^{n}}
\left(  \mathbb{E}_{\mathbb{P}}\left[  c_{i}^{2}\left(  Y\right)  \right]
\right)  ^{1/2}\left(  \mathbb{E}_{\mathbb{P}}\left[  \mathbb{P}^{2}\left\{
X_{i}\neq Y_{i}\mid Y\right\}  \right]  \right)  ^{1/2}\\
&  \leq\left(
{\displaystyle\sum\limits_{i=1}^{n}}
\mathbb{E}_{\mathbb{P}}\left[  c_{i}^{2}\left(  Y\right)  \right]  \right)
^{1/2}\left(
{\displaystyle\sum\limits_{i=1}^{n}}
\mathbb{E}_{\mathbb{P}}\left[  \mathbb{P}^{2}\left\{  X_{i}\neq Y_{i}\mid
Y\right\}  \right]  \right)  ^{1/2}\text{.}%
\end{align*}
We derive from this inequality that
\[
E_{Q}\left(  \zeta\right)  -E_{P^{n}}\left(  \zeta\right)  \leq\sqrt{v}\left(
\inf_{\mathbb{P\in}\mathcal{P}\left(  P^{n},Q\right)  }%
{\displaystyle\sum\limits_{i=1}^{n}}
\mathbb{E}_{\mathbb{P}}\left[  \mathbb{P}^{2}\left\{  X_{i}\neq Y_{i}\mid Y\right\}  \right]  \right)  ^{1/2}\text{,}%
\]
where $\mathcal{P}\left(  P^{n},Q\right)$ denotes the set all probability distributions $\mathbb P$, coupling $Q$ to $P^n$. Of course, exchanging the roles of $X$ and $Y$, a similar inequality holds for $-f$ instead of $\zeta$, more precisely 

\[
-E_{Q}\left(  \zeta\right)  +E_{P^{n}}\left(  \zeta\right)  \leq\sqrt{v}\left(
\inf_{\mathbb{P\in}\mathcal{P}\left(  P^{n},Q\right)  }%
{\displaystyle\sum\limits_{i=1}^{n}}
\mathbb{E}_{\mathbb{P}}\left[  \mathbb{P}^{2}\left\{  X_{i}\neq Y_{i}\mid X
\right\}  \right]  \right)  ^{1/2}.%
\]
Marton's following beautiful result tells us what happens when the coupling is chosen in a clever way. The version provided below (in which a symmetric conditioning with respect to $X$ and $Y$ is involved)  is due to Paul-Marie Samson (see \cite{Samson}). 
\begin{theorem}
\label{zmarton2}\textbf{(Marton's conditional transportation inequality)} Let
$P^{n}$ be some product probability distribution on some product measurable
space $\mathcal{X}^{n}$and $Q$ be some probability measure absolutely
continuous with respect to $P^{n}$. Then
\[
\min_{\mathbb{P}\in\mathcal{P}\left(  P^{n},Q\right)  }\mathbb{E}_{\mathbb{P}%
}\left[  \sum_{i=1}^{n}\mathbb{P}^{2}\left\{  X_{i}\neq Y_{i}\mid
X\right\}  +\mathbb{P}^{2}\left\{  X_{i}\neq Y_{i}\mid Y\right\}
\right]  \leq2D\left(  Q\parallel P^{n}\right)  \text{,}%
\]
where $\left(  X_{i},Y_{i}\right)  $, $1\leq i\leq n$ denote the coordinate
mappings on $\mathcal{X}^{n}\times\mathcal{X}^{n}$ and $D\left(  Q\parallel P^{n}\right)$ denotes the Kullback-Leibler divergence of $Q$ from $P^{n}$.
\end{theorem}

Now we can forget about the way the optimal coupling has been designed and focus on what gives us the combination between Theorem \ref{zmarton2} and the preceding inequalities. If we do so, we end up with the following inequality 
\begin{equation}
E_{Q}\left(  \zeta\right)  -E_{P^{n}}\left(  \zeta\right)  \leq \sqrt{2vD\left(  Q\parallel P^{n}\right)}\text{,} \label{fundconc}%
\end{equation}
which holds true for any probability distribution $Q$ which is absolutely continuous with respect to $P^n$ (the same inequality remaining true for $-\zeta$ instead of $\zeta$). It remains to connect this 
inequality with concentration, which can done thanks to a very simple but powerful engine: the variational formula for entropy. This 
formula is also well known in statistical mechanics, which is another domain of interest of Patrick Cattiaux. Let us briefly recall 
what this formula says for a random variable $\xi$ on some probability space $(\Omega , \mathcal A,P)$
\begin{equation}
\log E_{P}\left( e^{\xi}\right)=\sup_{Q\ll P} \left (E_{Q}(\xi) -D\left(  Q\parallel P\right)\right)\text{.} \label{varia}%
\end{equation}
How to use it? The trick is to rewrite (\ref {fundconc}) differently. Noticing that for any non negative real number $a$ 
\begin{equation}
\inf _{\lambda>0} \left(\frac{a}{\lambda}+\frac{\lambda v}{2}\right)=\sqrt{2av} \label{dualeinverse}
\end{equation}
and using (\ref{dualeinverse}) with $a=D(Q\parallel P^{n})$, inequality 
(\ref{fundconc}) means that for any positive $\lambda$
$$
\sup_{Q\ll P^n} [\lambda (E_{Q}(\zeta)-E_{P^{n}}(\zeta))- D(Q\parallel P^{n})] \leq \frac{\lambda^2v}{2} \text{.}
$$
It remains to combine this inequality with the variational formula (\ref{varia}) applied to the random variable 
$\xi =\lambda (f-E_{P^n}(f))$ to derive that for any positive $\lambda$
\[
\log E_{P^n}\left( e^{\lambda (\zeta-E_{P^n}(\zeta))}\right)\leq \frac{\lambda^2v}{2}. 
\]
Since, the same inequality holds true for $-\zeta$ instead of $\zeta$, this means that it actually holds for any real number $\lambda$. 
Applying Chernoff's inequality leads to the following concentration result around the mean which is the analogue of the preceding concentration inequality around the median apart from the fact that numerical constants are slightly different. 
\begin{corollary}
\label{ztail-marton}Let $\zeta$ satisfying regularity condition ($\mathcal {C}_v$), and $X_1,X_2,\dots,X_n$ be independent random variables. Setting $Z=\zeta(X_1,X_2,\dots ,X_n)$, then for all non negative real number $t$
$$P\{Z-EZ\leq -t\}\vee P\{Z-EZ\geq t\}\leq \exp (-t^2/(2v)).$$
\end{corollary}
Interestingly, as pointed out in \cite{BLMbook}, this result strictly implies Talagrand's convex distance inequality (and therefore 
Corollary \ref{medconc}). In other words, Marton's transportation inequality implies Talagrand's convex distance inequality.  
\subsection{The convex distance inequality revisited}
The key is that given any measurable subset $A$ of $\mathcal X^n$, Talagrand's convex distance $d_T(.,A)$ itself satisfies condition ($\mathcal C_1$). Indeed 
if $c\left(  x\right)  $ is a vector of $\mathcal B_n^+$ for which the supremum
in formula (\ref{dtal}) is achieved (which does exist since an upper semi-continuous function achieves a maximum on a compact set), we have%
\begin{align*}
d_{T}\left(  x,A\right)  -d_{T}\left(  y,A\right)    & \leq\inf_{x^{\prime}\in
A}\sum_{i=1}^{n}c_{i}\left(  x\right)  1\!\mathrm{l}_{x_{i}\neq x_{i}^{\prime
}}-\inf_{y^{\prime}\in A}\sum_{i=1}^{n}c_{i}\left(  x\right)  1\!\mathrm{l}%
_{y_{i}\neq y_{i}^{\prime}}\\
& \leq\sum_{i=1}^{n}c_{i}\left(  x\right)  1\!\mathrm{l}_{x_{i}\neq y_{i}}%
\end{align*}
with $\left\Vert \sum_{i=1}^{n}c_{i}^{2}\right\Vert _{\infty}\leq1$. This
means that $d_{T}\left(  .,A\right)  $ satisfies to condition ($\mathcal C_1$
) and Corollary \ref{ztail-marton} merely applies to $d_{T}\left(
X,A\right)$. It turns out
that this property strictly implies Talagrand's convex distance inequality.
Indeed, setting $Z=d_{T}\left(  .,A\right)  $ and $\theta=EZ$ by the right-tail bound provided by Corollary \ref{ztail-marton}
\[
P\left\{  Z-\theta\geq x\right\}  \leq\exp\left(  -\frac{x^{2}}{2}\right).
\]
Noticing that $x^{2}\geq-\theta^{2}+\left(  x+\theta\right)  ^{2}/2$, this
upper tail inequality a fortiori leads to
\begin{equation}
P\left\{  Z-\theta\geq x\right\}  \leq\exp\left(  \frac{\theta^{2}}{2}\right)
\exp\left(  -\frac{\left(  x+\theta\right)  ^{2}}{4}\right)  \text{.}%
\label{eright}%
\end{equation}
Setting $x=t-\theta$, this inequality can also imply that for positive $t$%
\[
P\left\{  Z\geq t\right\}  \leq\exp\left(  \frac{\theta^{2}}{2}\right)
\exp\left(  -\frac{t^{2}}{4}\right)
\]
(notice that this bound is trivial whenever $t\leq\theta$ and therefore we may
always assume that $t>\theta$ which warrants that $x>0$). On the other hand,
using the left-tail bound
\[
P\left\{  \theta-Z\geq x\right\}  \leq\exp\left(  -\frac{x^{2}}{2}\right)
\]
with $x=\theta$, we derive that%
\begin{equation}
P\left\{  X\in A\right\}  =P\left\{  Z=0\right\}  \leq\exp\left(
-\frac{\theta^{2}}{2}\right)  \text{.}\label{eleft}%
\end{equation}
Combining (\ref{eright}) with (\ref{eleft}) leads to%
\[
P\left\{  X\in A\right\}  P\left\{  Z\geq t\right\}  \leq\exp\left(
-\frac{t^{2}}{4}\right)
\]
which is precisely Talagrand's convex distance inequality.

\subsection{Application to Rademacher processes}
Recalling that a Rademacher variable is merely a uniformly distributed random 
sign, if $\epsilon_1,\epsilon_2,\ldots,\epsilon_n$ are independent Rademacher 
variables and if $B$ is a subset of $\mathbb R^n$, a Rademacher process is 
nothing else that $b\rightarrow \langle b,\epsilon \rangle$, where 
$\langle .,. \rangle$ denotes the canonical scalar product. The quantity of 
interest here is the supremum of such a process: 
$Z=\sup_{b\in B} \langle b,\epsilon \rangle$. 
One knows from Hoeffding's 
inequality (see \cite{Hoeff}) that for each given vector $b$ with Euclidean norm and all positive real number $t$
\begin{equation}
	P\{\langle b,\epsilon \rangle\geq t\}\leq \exp \left(-t^2/2\right). \label{ehoeff}
\end{equation} 
If $B$ is a subset of the unit closed Euclidean ball $\mathcal B_n$ of 
$\mathbb R^n$, and if one wants to prove a similar sub-Gaussian inequality for 
$Z-EZ$ and $EZ-Z$, where $Z=\sup_{b\in B} \langle b,\epsilon \rangle$, this a 
typical situation where the preceding theory applies. 
Consider the function $\zeta$, defined on $[-1,1]^n$ by 
$$\zeta:x \rightarrow \sup_{b\in B}\langle b,x \rangle.$$
Then $\zeta$ satisfies to condition ($\mathcal C_v$) with $v=4$. Indeed, possibly 
changing $B$ into its closure, one can always assume that $B$ is compact. 
Hence, there exists some point $b(x)$ 
belonging to $B$ such that $\zeta(x)=\langle b(x),x \rangle$ and therefore since $x$ 
and $y$ belong to $[-1,1]^n$
$$\zeta(x)-\zeta(y)\leq \langle b(x),x \rangle - \langle b(x),y \rangle=\langle b(x),x-y \rangle\leq 2 \sum_{i=1}^{n}\vert b_{i}\left(  x\right)\vert 1\!\mathrm{l}_{x_{i}\neq y_{i}},$$
which clearly means that $\zeta$ satisfies to ($\mathcal C_4$). Recalling that a Rademacher variable is merely a uniformly distributed random sign, applying Corollary \ref{ztail-marton} we derive the following concentration result for the supremum of a Rademacher process $Z=\sup_{b\in B}\langle b,\epsilon \rangle$.
\begin{proposition}
	
\label{rademacher}Let $B$ be some subset of the unit closed Euclidean ball of $\mathbb R^n$ and $\epsilon_1,\epsilon_2,\ldots,\epsilon_n$ be independent Rademacher random variables. Setting $Z=\sup_{b\in B}\langle b,\epsilon \rangle$, then for all non negative real number $t$
$$P\{Z-EZ\leq -t\}\vee P\{Z-EZ\geq t\}\leq \exp (-t^2/8).$$
\end{proposition}

Furthermore, the variance of the supremum of a Rademacher can easily be controlled via Efron-Stein's inequality. Considering independent Rademacher variables $\epsilon'_1,\epsilon'_2,\ldots,\epsilon'_n$ which are independent from $\epsilon_1,\epsilon_2,\ldots,\epsilon_n$ and setting $Z'_i=\sup_{b\in B}(b_i\epsilon'_i+\sum_{j\neq i} b_j\epsilon_j)$, Efron-Stein's inequality states that 
\newcommand{\Var}{\operatorname{Var}}
$$\Var(Z)\leq \sum_{i=1}^{n}E(Z-Z'_i)_+^2.$$
Now, writing $Z$ as $Z= \langle b(\epsilon),\epsilon \rangle$ and noticing that
$$Z-Z'_i \leq b_i(\epsilon)(\epsilon _{i}-\epsilon'_{i})$$
leads to 
$$E((Z-Z'_i)_+^2\vert \epsilon)\leq b_i^2(\epsilon)(1+\epsilon_i^2)$$
and finally to 
\begin{equation}
\Var(Z)\leq 2E\left(\sum_{i=1}^{n}b_i^2(\epsilon)\right)\leq 2. \label{es}
\end{equation}
The latter inequality is especially interesting to control the expectation of  the square root of a chi-square type statistics from 
below. 
More precisely, if we consider some orthonormal family of vectors $\{\phi_j ,1\leq j\leq D\}$ and if we define the chi-square type statistics $$\chi^2=\sum_{1\leq j\leq D}\langle \epsilon,\phi_j\rangle^2$$
$\chi$ can interpreted as the supremum of a Rademacher process. Indeed, if we simply set $B=\{\sum_{1\leq j\leq D}\theta_j\phi_j \vert \sum_{1\leq j\leq D}\theta_j^2\leq 1\}$, then 
$$\chi=\sup _{b\in B}\langle b, \epsilon\rangle.$$
Applying the above results to control the upper and lower tails of $\chi$ is exactly what we shall need in the statistical part of the paper to highlight phase transition phenomena in the behavior of penalized least squares model selection criteria. More precisely, since we know by (\ref{es}) that $\Var(\chi )\leq 2$, we derive the following sharp inequalities for the expectation of $\chi$
$$D-2\leq (E(\chi))^2\leq D.$$
Combining this with Proposition \ref{rademacher} leads to the following ready-to-use upper and lower tails controls, which hold for all positive $x$
\begin{equation}
	\chi \leq \sqrt{D}+2\sqrt{2x} \label{rightchi}
\end{equation}
except on a set with probability less than $e^{-x}$ while 
\begin{equation}
	\chi\geq \sqrt{(D-2)_{+}}-2\sqrt{2x} \label{leftchi}
\end{equation}
except on a set with probability less than $e^{-x}$.

\section{Model selection for regression with Rademacher errors}
Our aim is to show how some fairly general ideas (as those developed in \cite{BBM}, \cite{BM99g} or \cite{mas-stflour} for instance) work in a very simple context where the technical aspects are deliberately reduced. The statistical framework we have chosen is that of regression with Rademacher errors which can be described as follows. One observes 
\begin{equation}
Y=f+\sigma \epsilon \label{reg}
\end{equation}
where $f$ is some unknown vector in $\mathbb R^n$, $\epsilon$ is a random vector 
in $\mathbb R^n$ with components $\epsilon_1,\epsilon_2,\ldots,\epsilon_n$ which are 
independent Rademacher random variables and $\sigma$ is some positive real number 
(the level of noise, which is assumed to be known at this point). The issue is to 
estimate $f$ and the model selection approach to do so consists in starting from 
a (finite or countable) collection of models $\{S_m, m\in \mathcal M\}$ that we assume here to be 
linear subspaces of the Euclidean space $\mathbb R^n$. Consider for each model 
$S_m$ the least-square estimator which is merely defined as 
\begin{equation}
\hat f_m=\argmin_{g\in S_m}\parallel Y-g\parallel^2 \label{least-square}
\end{equation}
in other words, $\hat f_m$ is the orthogonal projection of $Y$ onto $S_m$.
The purpose is to select an estimator from the collection $\{\hat f_m, m\in 
\mathcal M\}$ in a clever way. We need some quality criterion here. Since we are 
dealing with least squares a natural one is the quadratic expected risk. Since 
everything is explicit, it is easy to compute it in this case. By Pythagoras' 
identity we indeed can decompose the quadratic loss of $\hat f_m$ as 
follows
$$\parallel \hat f_m-f\parallel^2=\parallel f_m-f\parallel^2+\parallel \hat f_m-f_m\parallel^2,$$
where $f_m$ denotes the orthogonal projection of $f$ onto $S_m$. The connection with the probabilistic part of the paper comes from the analysis of the random part of the decomposition. Denoting by $\Pi _m $ the orthogonal projection operator onto $S_m$ the random term can be written as 
\begin{equation}
	\parallel \hat f_m-f_m\parallel^2=\parallel \Pi _m(Y-f)\parallel^2=\sigma ^2 \parallel \Pi _m(\epsilon )\parallel^2. \label{randomterm}
\end{equation}

Taking some orthonormal basis $\{\phi ^{(m)}_j, 1\leq j\leq D_m\}$ of $S_m$ the quantity $\parallel \Pi _m(\epsilon )\parallel^2$ appears to be some chi-square type statistics 
$$ \chi^2_m=\parallel \Pi _m(\epsilon )\parallel^2=\sum_{1\leq j\leq D}\langle \epsilon,\phi^{(m)}_j\rangle^2$$
and therefore the expected quadratic risk of $\hat f_m$ can be computed as 
$$\mathbb E_f\parallel \hat f_m-f\parallel^2=\parallel f_m-f\parallel^2+\sigma ^2D_m$$

This formula for the quadratic risk perfectly reflects the model choice paradigm 
since if one wants to choose a model in such a way that the risk of the resulting 
least square estimator remains under control, we have to warrant that the
bias term $\parallel f_m-f\parallel^2$ and the variance term $\sigma ^2D_m$ remain 
simultaneously under control. This corresponds intuitively to what one should 
expect from a "good" model: it should fit to the data but should not be too 
complex in order to avoid overfitting. We therefore keep the quadratic risk as a 
quality criterion, which means that mathematically speaking, an
\textquotedblright ideal\textquotedblright\ model should minimize
$\mathbb{E}_f  \parallel \hat{f}_{m}-f\parallel ^{2}  $ 
with respect to $m\in\mathcal{M}$. It is called an 
"oracle". Of course, since we do not know the bias, 
the quadratic risk cannot be used as a statistical 
model choice criterion but just as a benchmark. The 
issue is now to consider data-driven criteria to 
select an estimator which tends to mimic an oracle, 
i.e. one would like the risk of the selected 
estimator $\hat{f}_{\hat{m}}$ to be as close 
as possible to the oracle benchmark 
$$\inf_{m\in \mathcal M}\mathbb{E}_f  \parallel \hat{f}_{m}-f \parallel ^{2} \text{.}$$

\subsection{Model selection via penalization and Mallows' heuristics}

Let us describe the method. The penalized least squares procedure consists in considering some proper penalty function
$\operatorname*{pen}$: $\mathcal{M}\rightarrow\mathbb{R}_{+}$ and take
$\hat{m}$ minimizing
$
  \parallel Y-\hat{f}_{m}\parallel^2  +\operatorname*{pen}\left(  m\right)
$
over $\mathcal{M}$. 
Since by Pythagora's identity, 
$$\parallel Y-\hat{f}_{m}\parallel^2 = \parallel Y\parallel^2-\parallel\hat f_m\parallel^2 \text{,}$$
we can equivalently consider $\hat{m}$ minimizing
$$
  -\parallel\hat{f}_{m}\parallel^2  +\operatorname*{pen}\left(  m\right)
$$
over $\mathcal{M}$.  Then, we can define the selected model $S_{\hat{m}}$
and the corresponding selected least squares estimator $\hat{f}_{\hat{m}}$.

Penalized criteria have been proposed in the early 
seventies by Akaike or Schwarz (see \cite{Akaike} 
and \cite{schwartz}) for penalized maximum
log-likelihood in the density estimation framework 
and Mallows for penalized least squares regression 
(see \cite{dan-wood} and \cite{Mallows}).
The crucial issue is: how to penalize? The classical answer given by Mallows' $C_{p}$ 
is based on some heuristics and on the unbiased risk estimation principle. It can be described as follows. An "ideal" model should 
minimize the quadratic risk%
\index{Quadratic risk}
\[
\left\Vert f_{m}-f\right\Vert ^{2}+\sigma^{2}D_{m}=\left\Vert
f\right\Vert ^{2}-\left\Vert f_{m}\right\Vert ^{2}+\sigma^{2}D_{m}\text{,}%
\]
or equivalently
\[
-\left\Vert f_{m}\right\Vert ^{2}+\sigma^{2}D_{m}\text{.}%
\]
At this step, it is tempting to use $\parallel \hat{f}_{m}\parallel ^{2}$ as an estimator of $\left\Vert f_{m}\right\Vert ^{2}$. But this estimator turns out to be biased. Indeed, starting from the decomposition
$$\parallel \hat{f}_{m}\parallel ^{2}-\left\Vert f_m\right\Vert ^{2}=\parallel \hat f_m-f_m\parallel^2+2\langle f_m,\hat{f}_{m}-f_m \rangle=\sigma ^2 \parallel\Pi_m(\epsilon)\parallel^2+2\sigma \langle f_m,\Pi_{m}(\epsilon ) \rangle$$
and noticing that by orthogonality $ \langle f_m,\Pi_{m}(\epsilon ) \rangle=\langle f_m,\epsilon  \rangle$,
leads to the following meaningful formula
\begin{equation}
	\parallel \hat{f}_{m}\parallel ^{2}=\parallel f_{m}\parallel ^{2}+\sigma^2\chi^2_m+2\sigma \langle f_m,\epsilon  \rangle \text{.} \label{riskestimate}
\end{equation}
From this formula we see that the expectation of $\parallel \hat{f}_{m}\parallel ^2$ is equal to $\parallel f_{m}\parallel ^{2}+\sigma^2D_m$. We can know remove this bias. Substituting to $\parallel f_{m}\parallel ^{2}$ its natural unbiased
estimator $\parallel \hat{f}_{m}\parallel ^{2}-\sigma^{2}D_{m}$
leads to
\index{Mallows' Cp@Mallows' $C_{p}$}%
Mallows' $C_{p}$
\[
-\parallel \hat{f}_{m} \parallel ^{2}+ 2 \sigma^{2}D_{m}\text{.}%
\]
The weakness of this analysis is that it relies on the computation of the
expectation of $\parallel \hat{f}_{m}\parallel ^{2}$ for every given
model but nothing warrants that $\parallel \hat{f}_{m}\parallel ^{2}$
will stay of the same order of magnitude as its expectation for all models
simultaneously. This leads to consider some more general model selection%
\index{Model selection!in the Gaussian framework@\textit{in the Gaussian framework}}
criteria involving penalties which may differ from Mallows%
\index{Mallows}%
' penalty.

\subsection{An oracle type inequality}

The above heuristics can be justified (or corrected) if one can specify how
close is $\parallel \hat{f}_{m}\parallel ^{2}$ from its expectation
$\left\Vert f_{m}\right\Vert ^{2}+\sigma^{2}D_{m}$, uniformly with
respect to $m\in\mathcal{M}$. The upper tail probability bound provided by Proposition \ref{rademacher} will
precisely be the adequate tool to do that. The price to pay is to consider more flexible penalty functions that can take into account 
the complexity of the list of models. As a consequence, the performance of the selected least-squares estimator  is judged by an oracle inequality that differs slightly from what might have been expected. The following result is the exact analogue of the model selection theorem established in \cite{BM99g} in the Gaussian regression framework.
\begin{theorem}
\label{modelselect}Let $\left\{  x_{m}\right\}  _{m\in\mathcal{M}}$ be some family
of positive numbers such that
\begin{equation}
\sum_{m\in\mathcal{M}}\exp\left(  -x_{m}\right)  =\Sigma<\infty\text{.}%
\label{EG2'}%
\end{equation}
Let $K>1$ and assume that
\begin{equation}
\operatorname{pen}\left(  m\right)  \geq K\sigma^{2}\left(  \sqrt{D_{m}%
}+2\sqrt{2x_{m}}\right)  ^{2}\text{.}\label{e3pen1}%
\end{equation}
Let $\hat{m}$ minimizing the penalized least-squares
criterion
\begin{equation}
	\operatorname{crit}(m)=-\parallel \hat{f}_{m}\parallel ^{2}+\operatorname{pen}\left(
m\right)  \label{pencrit}
\end{equation}
over $m\in\mathcal{M}$. The corresponding penalized least-squares estimator
$\hat{f}_{\hat{m}}$ satifies to the following risk bound
\begin{equation}
\mathbb{E}_{f}  \parallel \hat{f}_{\hat{m}}-f\parallel
^{2}  \leq C\left(  K\right)  \left\{  \inf_{m\in\mathcal{M}}\left(
\left\Vert f_{m}-f\right\Vert ^{2}+\operatorname{pen}\left(  m\right)
\right)  +\left(  1+\Sigma\right)  \sigma^{2}\right\}
\text{,}\label{e3riskoracle}%
\end{equation}
where $C\left(  K\right)  $ depends only on $K$.
\end{theorem}

The proof of this result is based on two claims. The first one provides a risk bound which derives from the very definition of the selection 
procedure via some elementary calculus while the second one is a consequence of the probabilistic material brought by the first part of 
the paper. Let us first introduce some notation. For all $m,m'\in \mathcal M$ we define 
\begin{equation}
	\chi _{m,m'}=\sup_{g\in S_{m'}} \frac{\langle g-f_m,\epsilon\rangle}{\left\Vert f_{m}-f\right\Vert+\left\Vert g-f\right\Vert} \label{defchimix}.
\end{equation}
The role of this supremum of a Rademacher process in the proof of Theorem \ref{modelselect} is elucidated by the following statement.
\begin{claim} \label{risketa}
	If $\hat m$ minimizes the penalized least-squares criterion (\ref{pencrit}), then for every $m\in \mathcal M$ and all $\eta \in ]0,1[$
	$$\eta \parallel \hat{f}_{\hat{m}}-f\parallel^{2}\leq \eta^{-1}\parallel f_{m}-f\parallel^2+\operatorname{pen}\left(
m\right)+\left (\frac{1+\eta}{1-\eta}\right )\sigma^2\chi _{m,\hat m}^2-\operatorname{pen}\left(
\hat m\right).$$
\end{claim}
\begin{proof} 
Pythagoras' identity combined with (\ref{riskestimate}) leads to
\begin{equation}
	\parallel f\parallel ^{2}+\operatorname{crit} (m)=\left\Vert f-f_{m}\right\Vert ^{2}-
	\sigma^2\chi^2_m-2\sigma \langle f_m,\epsilon  \rangle + 
	\operatorname{pen}\left(m\right) \text{.} \label{critexpression}
\end{equation}
Let $m$ be some given element of $\mathcal M$. By (\ref{critexpression}), $\operatorname{crit} (\hat m) \leq \operatorname{crit} (m)$ means that
$$\left\Vert f-f_{\hat m}\right\Vert ^{2}-\sigma^2\chi^2_{\hat m}\leq \left\Vert f-f_{m}\right\Vert ^{2}- \sigma^2\chi^2_m + \operatorname{pen}\left(m\right)+2\sigma \langle f_{\hat m}-f_m,\epsilon  \rangle-\operatorname{pen}\left(\hat m\right).$$
We can drop the non positive term $- \sigma^2\chi^2_m$ and add $2\sigma^2\chi^2_{\hat m}$ on both sides of the preceding  
preceding inequality, which leads to
$$\left\Vert f-f_{\hat m}\right\Vert ^{2}+\sigma^2\chi^2_{\hat m}\leq \left\Vert f-f_{m}\right\Vert ^{2} + \operatorname{pen}\left(m\right)+2\sigma \langle f_{\hat m}-f_m,\epsilon  \rangle+2\sigma^2\chi^2_{\hat m}-\operatorname{pen}\left(\hat m\right).$$
Noticing that $\sigma \chi^2_{\hat m}=\langle \hat f_{\hat m}-f_{\hat m},\epsilon  \rangle$ and therefore $ \langle f_{\hat m}-f_m,\epsilon  \rangle+\sigma \chi^2_{\hat m}= \langle \hat f_{\hat m}-f_m,\epsilon  \rangle$, we finally derive the inequality that we shall rely upon to prove Claim \ref	{risketa}
$$\left\Vert f-f_{\hat m}\right\Vert ^{2}+\sigma^2\chi^2_{\hat m}\leq \left\Vert f-f_{m}\right\Vert ^{2} + \operatorname{pen}\left(m\right)+2\sigma \langle \hat f_{\hat m}-f_m,\epsilon  \rangle-\operatorname{pen}\left(\hat m\right),$$
which yields
\begin{equation}
\parallel \hat{f}_{\hat{m}}-f\parallel^{2}\leq \parallel f-f_{m}\parallel ^{2} + \operatorname{pen}\left(m\right)+2\sigma \langle \hat f_{\hat m}-f_m,\epsilon  \rangle-\operatorname{pen}\left(\hat m\right). \label{riskbasic}
\end{equation}
To finish the proof, notice first that
$$2\sigma \langle \hat f_{\hat m}-f_m,\epsilon  \rangle \leq 2\sigma \left (\left\Vert f_{m}-f\right\Vert+\left\Vert \hat f_{\hat m}-f\right\Vert\right )\chi _{m,\hat m}.$$
Now we define $\delta = (1-\eta)/(1+\eta)$ and use repeatedly the inequality $2ab\leq a^2+b^2$ to derive that on the one hand
$$2\sigma \langle \hat f_{\hat m}-f_m,\epsilon  \rangle \leq \delta^{-1}\sigma ^2\chi _{m,\hat m}^2+ \delta \left(\left\Vert f_{m}-f\right\Vert+\left\Vert \hat f_{\hat m}-f\right\Vert\right )^2$$
and on the other hand 
$$\left(\left\Vert f_{m}-f\right\Vert+\left\Vert \hat f_{\hat m}-f\right\Vert\right )^2\leq (1+\eta ^{-1})\left\Vert f_{m}-f\right\Vert^2+(1+\eta)\left\Vert \hat f_{\hat m}-f\right\Vert^2 \text{.}$$
Combining these two inequalities and plugging the resulting upper bound on $2\sigma \langle \hat f_{\hat m}-f_m,\epsilon  \rangle$ into (\ref{riskbasic}) finally leads to the claim.
\end{proof}

Let us now state the second claim which will provide some control on the quantity $\chi _{m,m'}$ defined by (\ref{defchimix}).
\begin{claim} \label{controlchimix}
	For every $m,m' \in \mathcal M$, the following probability bound holds true. For all non negative real number $x$
$$\chi _{m,m'} \leq 1+\sqrt {D_{m'}}+2\sqrt{2x} $$
except on a set with probability less than $e^{-x}$.

\end{claim}
\begin{proof}
Since $\parallel g-f_m \parallel \leq \parallel f-f_m \parallel + \parallel g-f \parallel$ we can apply Proposition \ref{fundconc} and asserts that $$\chi _{m,m'} \leq  E(\chi _{m,m'})+2\sqrt{2x} $$
except on a set with probability less than $e^{-x}$. It remains to bound $E\left(\chi_{m,m'}\right )$. To do that we split the supremum defining $\chi _{m,m'}$ in two terms. Namely we set 
$$\chi _{m,m'}^{(1)}=\sup_{g\in S_{m'}} \frac{\langle g-f_{m'},\epsilon\rangle_{+}}{\left\Vert f_{m}-f\right\Vert+\left\Vert g-f\right\Vert} $$
and

$$\chi _{m,m'}^{(2)}=\sup_{g\in S_{m'}} \frac{\langle f_{m'}-f_m,\epsilon\rangle_{+}}{\left\Vert f_{m}-f\right\Vert+\left\Vert g-f\right\Vert} $$
noticing that $\chi _{m,m'}\leq \chi _{m,m'}^{(1)}+\chi _{m,m'}^{(2)}$. To control the first term, we note that since the orthogonal projection is a contraction, $\parallel g-f \parallel\geq \parallel g-f_{m'}\parallel $ for all $g\in S_{m'}$ and therefore by linearity
$$\chi _{m,m'}^{(1)}\leq \sup_{g\in S_{m'}} \frac{\langle g-f_{m'},\epsilon\rangle_{+}}{\left\Vert g-f_{m'}\right\Vert}=\sup_{g\in S_{m'}} \frac{\langle g,\epsilon\rangle}{\left\Vert g\right\Vert}=\chi_{m'} \text{.}  $$
Of course, this bound implies that $E\left(\chi_{m,m'}^{(1)}\right )\leq \sqrt{D_{m'}}$. To control the second term, we note by definition of $f_{m'}$ and the triangle inequality, that for all $g\in S_{m'}$
$$\left\Vert f_{m}-f\right\Vert + \parallel g-f \parallel \geq 
\left\Vert f_{m}-f\right\Vert+\parallel f_{m'}-f\parallel\geq 
\parallel f_{m'}-f_{m}\parallel$$
and therefore
$$\chi _{m,m'}^{(2)}\leq \frac{\langle f_{m'}-f_m,\epsilon\rangle_{+}}{\left\Vert f_{m'}-f_m\right\Vert} \text{.}$$
Invoking Cauchy-Schwarz, and using the fact that $a_{m,m'}=(f_{m'}-f_m)/\left\Vert f_{m'}-f_m\right\Vert$ has norm $1$, leads to
$$E\left (\chi _{m,m'}^{(2)}\right)\leq \sqrt{E\langle a_{m,m'}, \epsilon\rangle^2}=1 \text{.}$$
Collecting the upper bounds on the two terms $E\left (\chi _{m,m'}^{(1)}\right )$ and $E\left (\chi _{m,m'}^{(2)}\right )$ we get $E\left (\chi _{m,m'}\right)\leq 1+\sqrt {D_{m'}}$ and the proof is complete.
\end{proof}

Once these two claims are available, the proof of Theorem \ref{modelselect} is quite straightforward. 

\begin{proof}[Proof of Theorem \ref{modelselect}]
	
To prove the required bound on the expected risk, we first prove an exponential probability bound and then integrate it. Towards this aim we introduce some positive real number $\xi$ (this is the variable that we shall use at the end of the proof to integrate the tail bound that we shall obtain) and we fix some model $m\in \mathcal M$. Using a union bound, Claim \ref{controlchimix} ensures that for all $m'\in \mathcal M$ simultaneously

$$	\chi _{m,m'} \leq 1+\sqrt {D_{m'}}+2\sqrt{2(x_{m'}+\xi )} $$
except on a set with probability less than $\Sigma \exp(-\xi)$. Using $\sqrt{a+b}\leq \sqrt{a}+\sqrt{b}$ and using again $2ab\leq a^2+b^2$, if we define  
$$p_{m'}=\sigma^2\left (\sqrt{D_{m'}}+2\sqrt{2x_{m'}}\right)^2$$ the latter inequality implies that except on a set with probability less than $\Sigma \exp(-\xi)$
\begin{equation}
	\sigma^2\chi _{m,\hat m}^2\leq (1+\eta)p_{\hat m}+(1+\eta^{-1})\sigma^2\left(1+2\sqrt{2\xi}\right)^2 \text{.} \label{chideuxsimulane}
\end{equation}
Let us notice that the quantity $p_{m'}$ which appears here is precisely the one which is involved in the statement of Theorem \ref{modelselect} to bound the penalty from below. More precisely, this constraint can merely be written as $\operatorname{pen}(m')\geq Kp_{m'}$ for all $m'\in \mathcal M$. Let us by now choose $\eta$ in such a way that $K=(1+\eta)^2/(1+\eta)$, then the assumption on the penalty ensures that 
$$\left (\frac{1+\eta}{1-\eta}\right )(1+\eta)p_{\hat m}-\operatorname{pen}\left(\hat m\right)\leq 0.$$
Taking this constraint into account and combining (\ref{chideuxsimulane}) with Claim \ref{risketa} leads to
$$\eta \parallel \hat{f}_{\hat{m}}-f\parallel^{2}\leq \eta^{-1}\parallel f_{m}-f\parallel^2+\operatorname{pen}\left(
m\right)+\frac{(1+\eta)^2}{\eta(1-\eta)}\sigma^2\left(1+2\sqrt{2\xi}\right)^2  $$
except on a set with probability less than $\Sigma \exp(-\xi)$. Using a last time $2ab\leq a^2+b^2$ we upper bound $\left(1+2\sqrt{2\xi}\right)^2$ by $2+16\xi$ and it remains to integrate the resulting tail bound with respect to $\xi$ in order to get the desired upper bound on the expected risk. 
\end{proof}

It is interesting to exhibit some simple condition under which the above result can be applied to a choice of the penalty of the form $\operatorname{pen}\left(
m\right) = K'\sigma ^2D_m$, since obviously in this case the risk bound provided by the Theorem has the expected shape, that is, up to some constant the performance of the selected least-squares estimator is comparable to the infimum of the quadratic risks $\mathbb{E}_{f}  \parallel \hat{f}_{m}-f\parallel^{2} $ when $m$ varies in $\mathcal M$. This is connected to the possibility of choosing weights $x_m$ of the form $x_m=\alpha D_m$. The simplest scheme under which this can be done easily is the situation where the models are nested. In other words one starts from a family of linearly independent vectors $\phi_1, \phi_2,\ldots,\phi_N$ and each model $S_D$ with $1\leq D\leq N$ is merely defined as the linear span of $\phi_1, \phi_2,\ldots,\phi_D$. Indeed, in this case, since there is exactly one model per dimension, the choice $x_D=\alpha D$ leads to
$$\sum_{1\leq D \leq N}e^{-x_D}\leq \sum_{D\geq 1}e^{-\alpha D}=\frac{1}{e^\alpha-1}.$$
Choosing a sufficiently small value for $\alpha$ we finally derive from Theorem \ref{modelselect} that if the penalty is chosen as $\operatorname{pen}\left(
D\right) = \kappa\sigma ^2D$, with $\kappa >1$, for some constant $C'(\kappa)$ depending only on $\kappa$,
$$\mathbb{E}_{f}  \parallel \hat{f}_{\hat{D}}-f\parallel^{2}  
\leq C'(\kappa) \inf_{1\leq D\leq N} \mathbb{E}_{f} \parallel \hat{f}_{D}-f \parallel^{2}.$$
The same result will hold true if the number of models per dimension increases polynomially with respect to the dimension. The purpose of the following section is to show that in these situations if one takes a penalty of the form $\operatorname{pen}\left(m\right) = \kappa\sigma ^2D_m$, the value 
$\kappa =1 $ is indeed critical in the sense that, below this value the selection method becomes inconsistent. To enlighten this cut-off phenomenon, the lower tails probability bounds established in the section devoted to concentration will play a crucial role.

\subsection{Cut-off for the penalty: lower tails in action}\label{sec:Cut-off}
To exhibit this cut-off phenomenon for the penalty we shall restrict ourselves to the situation where all the models are included in a model $S_{m_N}$ with dimension $N$. We allow the list of models to depend on $N$ (we shall therefore write $\mathcal M_N$ instead of $\mathcal{M}$) and we shall let $N$ go to infinity, assuming  that the number of models is sub-exponential with respect to $N$, which more precisely means that 
\begin{equation}
	N^{-1}\log \#{\mathcal M_N}\rightarrow 0 \text{ as } N \rightarrow \infty \label{subexp}
\end{equation}
Note that in the nested case this assumption is satisfied and that it still holds true when the number of models with dimension $D$ is less that $CD^k$ since in this case $\#{\mathcal M_N}\leq CN^{k+1}$. We are now ready to state the announced negative result. This result has the same flavor as the one established in \cite{minpen} in the Gaussian framework but interestingly it is based solely on concentration arguments, without any extra properties (in \cite{minpen} the Gaussian framework is crucially involved since some specific lower tail bounds for non-central chi-square distributions are used to make the proof).
\begin{theorem}\label{theo:Cut-off}
	Let $\{S_m, m\in \mathcal M_N\}$ be a collection of linear subspaces of $\mathbb R^n$ such that all the models $S_m$ are included in some model $S_{m_N}$ with dimension $N$. Assume furthermore that condition (\ref{subexp}) on the cardinality of $\mathcal M_N$ is satisfied. Take a penalty function of the form 
	$$\operatorname{pen}\left(m\right) = \kappa\sigma ^2D_m$$
	and consider $\hat m$ minimizing the penalized least squares criterion (\ref{pencrit}). Assume that $\kappa<1$. Then, for any $\delta\in (0,1)$ there exists $N_0$ depending on $\delta$ and $\kappa$ but not on $f$ or $\sigma$ such that, whatever $f$, for all $N\geq N_0$ 
	\begin{equation}
		\mathbb P_f\{D_{\hat m}\geq N/2\}\geq 1-\delta \label{largemodel}
	\end{equation}
and the following lower bound on the expected risk holds true
	\begin{equation}
		\mathbb{E}_{f}  \parallel \hat{f}_{\hat{m}}-f\parallel^{2} \geq \parallel f_{m_N}-f\parallel^{2}+\sigma^2(N/4).\label{largerisk}
	\end{equation}
\end{theorem}

\begin{proof}
	Let $m\in \mathcal M_N$ and first notice that since $S_m\subseteq S_{m_N}$, by orthogonality $-\parallel f_{m_N}\parallel^2+\parallel f_{m}\parallel^2=-\parallel f_{m_N}-f_m\parallel^2$. Let us now use formula (\ref{riskestimate}) to assert that
	$$-\parallel\hat f_{m_N} \parallel ^2+\parallel\hat f_{m} \parallel ^2=-\sigma^2(\chi^2_{m_N}-\chi^2_{m})+2\sigma\langle f_m-f_{m_N},\epsilon\rangle-\parallel f_{m_N}-f_m\parallel^2 \text{.}$$
	Let us set $g_m=(f_m-f_{m_N})/\parallel f_m-f_{m_N}\parallel$ if $f_m\neq f_{m_N}$ or $g_m=0$ otherwise. Using again $2ab\leq a^2+b^2$, the preceding identity leads to 
	$$-\parallel\hat f_{m_N} \parallel ^2+\parallel\hat f_{m} \parallel ^2\leq -\sigma^2(\chi^2_{m_N}-\chi^2_{m})+\sigma^2\langle g_m,\epsilon\rangle_{+}^2 \text{.}$$
	Using the definition of the penalized least-squares criterion, we finally derive the inequality that we shall start from to make the probabilistic analysis of the behavior of this criterion:
	\begin{equation}
		\sigma^{-2}(\operatorname{crit}(m_N)-\operatorname{crit}(m))\leq -(\chi^2_{m_N}-\chi^2_{m})+\langle g_m,\epsilon\rangle_{+}^2+\kappa(N-D_m) \text{.} \label{lowpen} 
	\end{equation}
	The point now is that for models such that $D_m\leq N/2$ the negative term 
	$-(1-\kappa)(N-D_m)$ stays below $-(1-\kappa)N/2$. We argue that this is 
	enough to ensure that, with high probability, for all such models 
	simultaneously $\operatorname{crit}(m_N)-\operatorname{crit}(m)<0$ and 
	therefore $D_{\hat m}$ has to be larger than $N/2$. To complete this road map 
	we make use of the lower tail probability bounds established in the 
	probabilistic section of the paper. Indeed, since $S_m\subseteq S_{m_N}$, the 
	quantity $\chi^2_{m_N}-\chi^2_{m}$ appears to be some pseudo chi-square 
	statistics with dimension $N-D_m\geq N/2$ to which we can apply the lower 
	tail inequality (\ref{leftchi}). If we do so, and if we use a union bound, we 
	derive that for all models such that $D_m\leq N/2$ 
	$$\sqrt{\chi^2_{m_N}-\chi^2_{m}}\geq \sqrt {(N-D_m-2)_+}-2\sqrt{2x}$$ 
while simultaneously by Hoeffding's inequality (\ref{ehoeff})
	$$\langle g_m,\epsilon\rangle_+\leq \sqrt{{2x}}$$
except on a set with probability less than $2\#{\mathcal M_N}\exp({-x})$. We choose $x=\log(2/\delta)+\log\# \mathcal M_N$ in order to warrant that the above inequalities simultaneously hold true except on a set with probability less than $\delta$. It is now time to use asymptotic arguments. If we take into account assumption (\ref{subexp}), we know that our choice of $x$ is small as compared to $N$ and therefore it is also small as compared to $N-D_m$ uniformly over the set of models such that $D_m\leq N/2$ when $N$ is large. Using this argument we derive from the above tail probability bounds that given $\eta>0$, if $N$ is large enough, the following inequalities hold for all models such that $D_m\leq N/2$
$$\sqrt{\chi^2_{m_N}-\chi^2_{m}}\geq \sqrt {(1-\eta)(N-D_m)}$$
and 
$$\langle g_m,\epsilon\rangle_+\leq \sqrt{\eta(N-D_m)}$$
except on a set with probability less than $\delta$. 
If $N$ is large enough, plugging these inequalities into (\ref{lowpen}) and choosing $\eta=(1-\kappa)/4$ leads to 
$$\sigma^{-2}(\operatorname{crit}(m_N)-\operatorname{crit}(m))\leq -(1-2\eta-\kappa)(N-D_m)\leq -\Big(\frac{1-\kappa}{2}\Big)(N-D_m)$$
for all models such that $D_m\leq N/2$, except on a set with probability less than $\delta$. The proof of (\ref{largemodel}) is now complete. Proving (\ref{largerisk}) is quite easy. We first observe that since $S_{m_N}$ includes all the other models
$$\parallel \hat{f}_{\hat{m}}-f\parallel^{2} \geq \parallel f_{m_N}-f\parallel^{2}+\parallel \hat{f}_{\hat{m}}-f_{\hat m}\parallel^{2}=\parallel f_{m_N}-f\parallel^{2}+\sigma^2 \chi^2_{\hat m} $$ 
so that it remains to bound  $\chi^2_{\hat m}$ in expectation from below. To do that we argue exactly as above to assert that if $N$ is large enough, for all models such that $D_m\geq N/2$ simultaneously
$$\chi^2_m\geq (2/3)D_m\geq N/3$$
except on a set with probability less than $\delta$. Combining this with (\ref{largemodel}) and using again a union bound argument we know that if $N$ is large enough $\chi^2_{\hat m}\geq N/3$ except on a set with probability less than $2\delta$. It remains to use Markov's inequality and choose $\delta=1/8$ to ensure that 
$$\mathbb E\left(\chi^2_{\hat m}\right)\geq (1-2\delta)N/3=N/4$$
completing the proof of (\ref{largerisk}).
\end{proof}
\\
\textbf{Comment}\\
Let us come back to the nested case for which one starts from a set of linearly independent vectors $\phi_1,\phi_2,\ldots,\phi_N$ and a model is merely the linear span $S_D$ of $\{\phi_j, 1\leq j \leq D \}$ and $D$ varies between $1$ and $N$. In this case the situation is clear. If one considers the penalized least squares model selection criterion 
$$\operatorname{crit}(D)=-\parallel\hat f_D\parallel^2+\kappa \sigma^2D$$
the two preceding theorems tell us that $\kappa =1$ is a critical value in the sense that if $\kappa$ is above this value the selected least squares estimator is comparable (up to some constant depending on $\kappa$) to the best estimator in the collection while below this value the criterion will tend to select the largest models whatever the target $f$ to be estimated. This cut-off is so visible (on simulations and on real data) that it can be used to estimate $\sigma^2$. Of course, the notion of a “large” model only makes sense if $N$ is large (and thus so is $n$). In the next and final section of the article, we shall see that this framework becomes very natural in the context of non-parametric estimation.
\subsection{Adaptive functional estimation}\label{sec:functional}
In this section, the goal is to estimate the function $f$ on the interval $[0,1]$  in the model
\begin{equation}\label{regf}
Y_k=f(t_k)+\sigma \epsilon_k,\quad k=1,\ldots,n, \end{equation}
where $t_k=k/n$ and we assume in the sequel that the $\epsilon_k$'s are 
independent Rademacher random variables but our results remain valid if they are only centered i.i.d. bounded random variables;  the noise level, $\sigma>0$, is assumed to be known. We shall also assume that $f$ is squared integrable and  $f(0)=f(1)$, so that we can expand $f$ on the Fourier basis $(\phi_j)_{j\geq 1}$ with $\phi_1\equiv 1$ and for any $j\geq 1$ and any $t\in [0,1]$,
$$\phi_{2j}(t)=\sqrt{2}\cos(2\pi jt)\quad\mbox{and}\quad \phi_{2j+1}(t)=\sqrt{2}\sin(2\pi jt).$$
Denoting $\theta=(\theta_j)_{j\geq 1}$ the sequence of the Fourier coefficients of $f$, we obtain:
$$f=\sum_{j=1}^{+\infty}\theta_j\phi_j.$$
We derive oracle inequalities in the same spirit as Theorem~\ref{modelselect}, except that we consider both the empirical norm associated with the design $t_k$'s and the functional ${\mathbb L}_2$-norm. We then introduce following notations: for any function $g$, we set:
$$\normn{g}^2=\frac{1}{n}\sum_{k=1}^ng^2(t_k),\quad \norm{g}^2=\int_0^1g^2(t)dt.$$
The associated scalar products are denoted $\langle\cdot,\cdot\rangle_n$ and $\langle\cdot,\cdot\rangle_{{\mathbb L}_2}$. 
We recall that the Fourier basis satisfies for any $1\leq j,j'\leq n-1$,
\begin{equation}\label{FourierDesign}
\langle\phi_j,\phi_{j'}\rangle_n=\frac{1}{n}\sum_{k=1}^n\phi_j(t_k)\phi_{j'}(t_k)=1_{\{j=j'\}},
\end{equation}
which makes its use suitable for our study. We consider a collection of models $\{S_m, m\in \mathcal M\}$,
with here
$$S_m=\mbox{span}\big\{\phi_j, j\in m\big\},$$
with $\mathcal M$ a set of subsets of $\{1,2,\ldots,n-1\}$. Similarly to \eqref{pencrit}, we consider for any $m\in\mathcal{M}$
the criterion
\begin{equation}\label{pencrit2}
	\operatorname{crit}(m)=-\normn{\hat{f}_{m}}^2+\operatorname{pen}\left(
m\right),
\end{equation}
with
$$ \hat{f}_{m}=\sum_{j\in m}\frac{1}{n}\sum_{k=1}^nY_k\phi_j(t_k)\phi_j.$$
Observe that if $Y$ is any (random) $1$-periodic ${\mathbb L}_2$-function such that $Y(t_k)=Y_k$, then $ \hat{f}_{m}$ is the projection of the function $Y$ onto $S_m$ for the empirical norm~$\normn{\cdot}:$
$$\hat{f}_{m}=\sum_{j\in m}\langle Y,\phi_j\rangle_n\phi_j.$$
Therefore, 
$$\normn{Y}^2-\normn{\hat{f}_{m}}^2=\normn{Y-\hat{f}_{m}}^2=\frac{1}{n}\sum_{k=1}^n\Big(Y_k-\hat{f}_{m}(t_k)\Big)^2,$$
which justifies the use of \eqref{pencrit2}. Observe also that $f_m$, the mean of $\hat{f}_{m}$, satisfies
$$f_m=\mathbb{E}_{f} (\hat{f}_{m})=\sum_{j\in m}\langle f,\phi_j\rangle_n\phi_j$$
and $f_m$ is the orthogonal projection of $f$ on $S_m$ for the empirical norm.
The following result is the analogue  of Theorem~\ref{modelselect} in the functional framework.
We denote
$${\mathcal S}_{n-1}=\mbox{span}(\phi_1,\ldots,\phi_{n-1}).$$
\begin{theorem}\label{modelselect2}
Let $\left\{  x_{m}\right\}  _{m\in\mathcal{M}}$ be some family
of positive numbers such that
\begin{equation}
\sum_{m\in\mathcal{M}}\exp\left(  -x_{m}\right)  =\Sigma<\infty\text{.}%
\label{EG2f'}%
\end{equation}
Let $K>1$ and assume that
\begin{equation}
\operatorname{pen}\left(  m\right)  \geq \frac{K\sigma^{2}}{n}\left(  \sqrt{D_{m}%
}+2\sqrt{2x_{m}}\right)  ^{2}\text{.}\label{e3pen1f}%
\end{equation}
Let $\hat{m}$ minimizing the penalized least-squares
criterion defined in \eqref{pencrit2} over $m\in\mathcal{M}$. The corresponding penalized least-squares estimator
$\hat{f}_{\hat{m}}$ satifies to the following risk bound
\begin{equation}
\mathbb{E}_{f} \normn{\hat{f}_{\hat{m}}-f}
^{2}  \leq C\left(  K\right)  \left\{  \inf_{m\in\mathcal{M}}\Big(
\normn{f_{m}-f}^{2}+\operatorname{pen}\left(  m\right)
\Big)+\frac{\left(  1+\Sigma\right)  \sigma^{2}}{n}\right\}
\text{,}\label{e4riskoracle}%
\end{equation}
where $C\left(  K\right)  $ depends only on $K$. We also have:
\begin{equation}
\mathbb{E}_{f} \norm{\hat{f}_{\hat{m}}-f}
^{2}  \leq C'\left(  K\right)  \left\{  \inf_{m\in\mathcal{M}}\Big(
\norm{f_{m}-f}^{2}+\operatorname{pen}\left(  m\right)
\Big)  +\inf_{g\in {\mathcal S}_{n-1}}\|f-g\|^2_{{\mathbb L}_{\infty}}+\frac{\left(  1+\Sigma\right)  \sigma^{2}}{n}\right\}
\text{,}\label{e5riskoracle}%
\end{equation}
where $C'\left(  K\right)  $ depends only on $K$ and $\|\cdot\|_{{\mathbb L}_{\infty}}$ denotes the sup-norm on $[0,1]$.
\end{theorem}
\begin{proof}
We observe that
$$\normn{\hat{f}_{m}}^2=\sum_{j\in m}\langle Y,\phi_j\rangle_n^2.$$
To prove the first point of Theorem~\ref{modelselect2}, we then follow easily the same lines as used to prove Theorem~\ref{modelselect} with
$$\chi^2_m=\frac{1}{\sigma^2}\|\hat f_m-f_m\|_n^2.$$
Proposition~\ref{fundconc}  is applied with
$$
\chi _{m,m'}=\sup_{g\in S_{m'}} \frac{\langle g-f_m,Y-f\rangle_n}{\left\Vert f_{m}-f\right\Vert_n+\left\Vert g-f\right\Vert_n}.
$$

The last point is a simple consequence of \eqref{e4riskoracle} and
$$\norm{g}=\normn{g}$$
for any function $g\in {\mathcal S}_{n-1}$. Indeed, we have, for any $g\in {\mathcal S}_{n-1}$,
\begin{align*}
\norm{\hat{f}_{\hat{m}}-f}&\leq \norm{\hat{f}_{\hat{m}}-g}+\norm{f-g}\\
&\leq \normn{\hat{f}_{\hat{m}}-g}+\|f-g\|_{{\mathbb L}_{\infty}}\\
&\leq \normn{\hat{f}_{\hat{m}}-f}+2\|f-g\|_{{\mathbb L}_{\infty}}
\end{align*}
and
\begin{align*}
\normn{f_{m}-f}&\leq \normn{f_{m}-g}+\normn{f-g}\\
&\leq \norm{f_{m}-g}+\|f-g\|_{{\mathbb L}_{\infty}}\\
&\leq \norm{f_{m}-f}+2\|f-g\|_{{\mathbb L}_{\infty}}.
\end{align*}
We have used that $\normn{f-g}\leq \|f-g\|_{{\mathbb L}_{\infty}}$ and $\norm{f-g}\leq \|f-g\|_{{\mathbb L}_{\infty}}$.
\end{proof}

To prove optimality of our procedure, we consider the  minimax setting and establish rates of our estimate on the class of (periodized) Sobolev spaces. We recall the definition of the Sobolev ball for integer smoothness $\alpha$.
\begin{definition}
Let $\alpha\in\{1,2,\ldots\}$ and $R>0$. The Sobolev ball $W(\alpha,R)$ is defined by
\begin{align*}
W(\alpha,R)&=\Bigg\{g\in[0,1]\longmapsto{\mathbb R}:\ g^{(\alpha-1)} \mbox{ is absolutely continuous and }\\
&\hspace{6cm} \int_0^1\big(g^{(\alpha)}(x)\big)^2dx\leq R^2\Bigg\}.
\end{align*}
\end{definition}
In our setting, we consider the periodic Sobolev ball $W^{per}(\alpha,R)$ defined by
\begin{align*}
W^{per}(\alpha,R)&=\Bigg\{g\in W(\alpha,R):\ g^{(j)}(0)=g^{(j)}(1), \ j=0,1,\ldots,\alpha-1\Bigg\}.
\end{align*}
In subsequent Theorem~\ref{rates}, we consider the model selection procedure with $\mathcal M$ such that $m\in \mathcal M$ if and only if $m$ is of the form $m=\{1,\ldots,D\}$ for some $1\leq D\leq n-1$. In this case, $D_m=D$. Applying Theorem~\ref{modelselect2} with $x_m=xD_m$ for any arbitrary constant $x$
and 
$$\operatorname{pen}\left(  m\right) = \frac{K\sigma^{2}}{n}\left(  \sqrt{D_{m}%
}+2\sqrt{2x_{m}}\right)  ^{2},$$
for some constant $K>1$, we obtain:
\begin{theorem}\label{rates}
Let $\alpha\geq 1$ and $R>0$. Then, we have:
$$\sup_{f\in W^{per}(\alpha,R)}E\norm{\hat{f}_{\hat{m}}-f}^2\leq Cn^{-\frac{2\alpha}{2\alpha+1}},$$
where $C$ depends on $\sigma$, $\alpha$ and $R$.
\end{theorem}
It can be proved by using standard arguments that
$$\liminf_{n\to+\infty}\inf_{T_n} \sup_{f\in W^{per}(\alpha,R)}E\Big[n^{\frac{2\alpha}{2\alpha+1}}\norm{T_n-f}^2\Big]\geq \tilde C,$$
where $\inf_{T_n}$ denotes the infimum over all estimators and where the constant $\tilde C$ depends on $\sigma$, $\alpha$ and $R$. Therefore, the previous theorem shows that $\hat{f}_{\hat{m}}$ achieves the optimal  minimax rate. It is also adaptive since it does not depend on the parameters $\alpha$ and $R$ which are unknown in practice.

\medskip

\begin{proof}
The set $W^{per}(\alpha,R)$ can be characterized by Fourier coefficients and by using the following proposition established in \cite{Tsybakov}:
\begin{proposition}\label{ellipsoid:prop}
Let $\alpha\in\{1,2,\ldots\}$ and $R>0$. Then, the function $f$ belongs to $W^{per}(\alpha,R)$ if and only if the sequence of its Fourier coefficients $\theta=(\theta_j)_{j\geq 1}$ belongs to the ellipsoid $\Theta(c,r)$ defined by
$$\Theta(c,r)=\Bigg\{\theta\in \ell_2:\ \sum_{j=1}^{+\infty}c_j^2\theta_j^2\leq r^2\Bigg\},$$
with $r=R/\pi^{\alpha}$ and
$$c_j=\left\{
\begin{array}{cll}
 j^{\alpha} &\mbox{ if}   &j \mbox{ is even,}  \\
  (j-1)^{\alpha}  &\mbox{ if}    & j\mbox{ is odd.}   
  \end{array}
\right.$$
\end{proposition}
Now, we use Inequality \eqref{e5riskoracle} of Theorem~\ref{modelselect2}. Let $m\in{\mathcal M}$ be fixed. Proposition~\ref{ellipsoid:prop} allows to control the bias term:
\begin{align*}
\norm{f_{m}-f}^{2}&=\sum_{j\in m}\big(\langle f,\phi_j\rangle_n-\theta_j\big)^2+\sum_{j\not\in m}\theta_j^2.
\end{align*}
For the first term, we have for any $j\in m$,
\begin{align*}
\langle f,\phi_j\rangle_n-\theta_j&=\frac{1}{n}\sum_{k=1}^n\sum_{i=1}^{+\infty}\theta_i\phi_i(t_k)\phi_j(t_k)-\theta_j\\
&=\sum_{i=1}^{n-1}\theta_i\times\frac{1}{n}\sum_{k=1}^n\phi_i(t_k)\phi_j(t_k)-\theta_j+\frac{1}{n}\sum_{k=1}^n\sum_{i=n}^{+\infty}\theta_i\phi_i(t_k)\phi_j(t_k)\\
&=\frac{1}{n}\sum_{k=1}^n\sum_{i=n}^{+\infty}\theta_i\phi_i(t_k)\phi_j(t_k).
\end{align*}
Thus,
$$\max_{j\in m}|\langle f,\phi_j\rangle_n-\theta_j|\leq 2 \sum_{i=n}^{+\infty}|\theta_i|$$
and
\begin{align*}
\norm{f_{m}-f}^{2}&=\sum_{j\in m}\big(\langle f,\phi_j\rangle_n-\theta_j\big)^2+\sum_{j\not\in m}\theta_j^2\\
&\leq 4D_m\Big(\sum_{i=n}^{+\infty}|\theta_i|\Big)^2+\sum_{j=D_m+1}^{+\infty}\theta_j^2\\
&\leq 4D_m\times\sum_{i=1}^{+\infty} c_i^2\theta_i^2\times\sum_{i\geq n}c_i^{-2}+D_m^{-2\alpha}\sum_{j=1}^{\infty}c_j^2\theta_j^2\\
&\leq c_{\alpha, R}\Big(D_mn^{-2\alpha+1}+D_m^{-2\alpha}\Big),
\end{align*}
with $c_{\alpha, R}$ only depending on $\alpha$ and $R$. We have used $\alpha>1/2$. We also have: \begin{align*}
\inf_{g\in {\mathcal S}_{n-1}}\|f-g\|_{{\mathbb L}_{\infty}}&\leq  \Big\|\sum_{i\geq n}\theta_i\phi_i\Big\|_{{\mathbb L}_{\infty}}\leq\sqrt{2}\sum_{i\geq n}|\theta_i|\leq \sqrt{2}\Big(\sum_{i=1}^{+\infty} c_i^2\theta_i^2\times\sum_{i\geq n}c_i^{-2}\Big)^{1/2}.
\end{align*}
Finally,
\begin{align*}
\inf_{g\in {\mathcal S}_{n-1}}\|f-g\|_{{\mathbb L}_{\infty}}&\leq c'_{\alpha, R}n^{-\alpha+1/2},
\end{align*}
with $c'_{\alpha, R}$ only depending on $\alpha$ and $R$. To conclude, we observe that
\begin{align*}
&\inf_{m\in\mathcal{M}}\Big(
\norm{f_{m}-f}^{2}+\operatorname{pen}\left(  m\right)
\Big)  +\inf_{g\in {\mathcal S}_{n-1}}\|f-g\|^2_{{\mathbb L}_{\infty}}+\frac{\left(  1+\Sigma\right)  \sigma^{2}}{n}\\
&\leq C\inf_{1\leq D_m\leq n-1}\Bigg\{D_mn^{-2\alpha+1}+D_m^{-2\alpha}+\frac{D_m\sigma^2}{n}\Bigg\}+n^{-2\alpha+1}+\frac{\left(  1+\Sigma\right)  \sigma^{2}}{n},
\end{align*}
with $C$ depending on $\alpha$ and $R$. We take $D_m\in \{1,\ldots, n-1\}$ of order $(n/\sigma^2)^{1/(2\alpha+1)}$ to conclude. Observe that the assumption $\alpha\geq 1$ allows to state that the term $D_mn^{-2\alpha+1}$ is smaller than $D_m^{-2\alpha}\vee\frac{D_m\sigma^2}{n}$, up to a constant.
\end{proof}

We end this section by deriving the cut-off phenomenon for the penalty in the functional setting. Even if the analogous general results of Section~\ref{sec:Cut-off} can be obtained, we only consider the case where the collection of models ${\mathcal M}$ is the following: a model $m\in {\mathcal M}$ if and only if it is of the form $m=\{1,\ldots,d\}$ for some $d\in\{1,\ldots,n-1\}$. In particular, all models are nested and $\#{\mathcal M}=n-1$. For sake of simplicity, we further assume that $f\in{\mathcal S}_{n-1}$.
Mimicking the proof of Theorem~\ref{theo:Cut-off}, we obtain:
\begin{theorem}\label{theo2:Cut-off}
Take a penalty function of the form 
	$$\operatorname{pen}\left(m\right) = \frac{\kappa\sigma ^2D_m}{n}$$
	and consider $\hat m$ minimizing the penalized least squares criterion (\ref{pencrit2}). Assume that $\kappa<1$. Then, for any $\delta\in (0,1)$ there exists $N_0$ depending on $\delta$ and $\kappa$ but not on $f$ or $\sigma$ such that, whatever $f\in{\mathcal S}_{n-1}$, for all $n\geq N_0$ 
	\begin{equation*}
		\mathbb P_f\{D_{\hat m}\geq n/2\}\geq 1-\delta
	\end{equation*}
and the following lower bounds on the expected risks hold true
	\begin{equation*}
		\mathbb{E}_{f}  \normn{\hat{f}_{\hat{m}}-f}^{2} \geq\frac{\sigma^2}{4},\quad\mathbb{E}_{f}  \norm{\hat{f}_{\hat{m}}-f}^{2} \geq\frac{\sigma^2}{4}.
	\end{equation*}
\end{theorem}
Some simulations are carried out in Figure~\ref{graph} to illustrate  this last result in the non-asymptotic setting. More precisely, in the framework of Model~\eqref{regf} with $\sigma=1$ and $n=100$, we consider the estimation of the function $f(x)=2+0.7\sqrt{2}\cos(2\pi x)+0.5\sqrt{2}\sin(2\pi x)$, which brings this problem in the setting of Theorem~\ref{theo2:Cut-off}. Note in particular that $f$ belongs to $S_m$, with $m=\{1,2,3\}$. The graph of the left hand side provides the value of $D_{\hat m}$ with respect to $\kappa$, where $\kappa$ is the constant involved in the penalty function $\operatorname{pen}$ of Theorem~\ref{theo2:Cut-off}. We observe a jump around the value $\kappa=1$, as predicted by the theory, with in particular very large models being selected when $\kappa<1$. Observe that true model is selected ($D_{\hat m}=3$), as soon as $\kappa\geq 1.3$. On the right hand of Figure~\ref{graph}, we display the value of $\|\hat{f}_{m}\|^2$ with respect to $D_m$. Once $D_m$ is larger or equal to 3, this function is approximately linear and the estimation of the slope of the linear part of the curve is equal to $\hat\kappa\times\sigma^2/n$ with $\hat\kappa=0.988$.
\begin{figure}[!h]
	\begin{center}
		\includegraphics[scale=0.32]{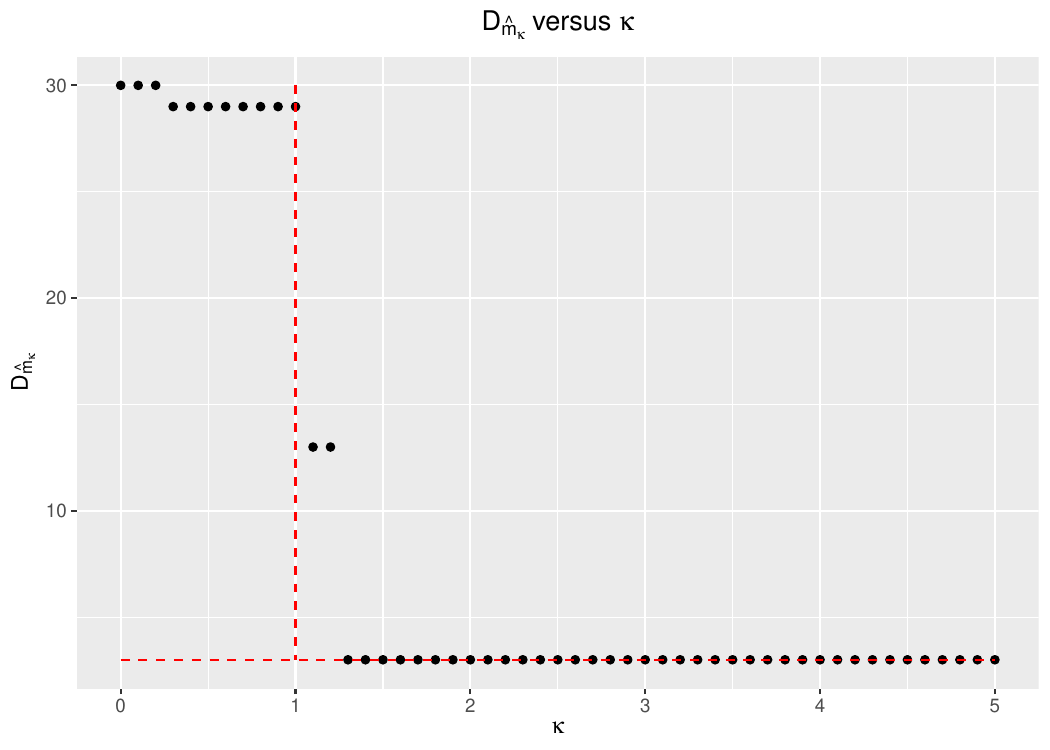}
		\includegraphics[scale=0.32]{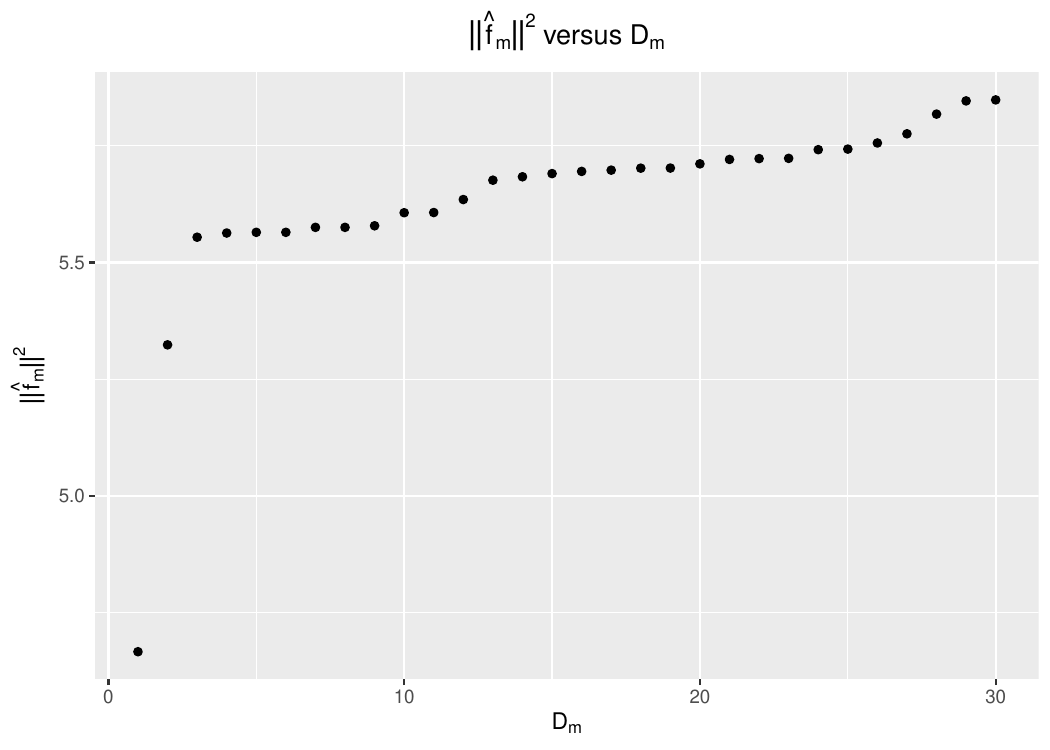}
	\end{center}
	\caption{Estimation of the function $f(x)=2+0.7\sqrt{2}\cos(2\pi x)+0.5\sqrt{2}\sin(2\pi x)$ in  Model~\eqref{regf} with $\sigma=1$ and $n=100$ in the setting of Theorem~\ref{theo2:Cut-off}. Left hand side: graph of $\kappa\longmapsto D_{\hat m}$. Right hand side:  graph of $D_m\longmapsto \|\hat{f}_{m}\|^2$.}\label{graph}
\end{figure}

\subsection*{Acknowledgements}
The authors are very grateful to Suzanne Varet, who carried out the numerical study of Section~\ref{sec:functional} and produced the graphs of Figure~\ref{graph}.



\end{document}